\documentclass[11pt]{amsart}

\input{xypic}

\newtheorem{theorem}{Theorem}[section]
\newtheorem{lemma}[theorem]{Lemma}
\newtheorem{corollary}[theorem]{Corollary}
\newtheorem{proposition}[theorem]{Proposition}
\theoremstyle{remark}
\newtheorem{remark}[theorem]{Remark}
\theoremstyle{remark}

\newcommand\md{{\text{\rm mod}}} 
\newcommand\Ab{{{\mathcal A}b}} 

\newcommand{\Hom}{\mathop{\rm Hom}\nolimits} 
\newcommand{\Ker}{\mathop{\rm Ker}\nolimits} 
\newcommand{\ima}{\mathop{\rm im}\nolimits} 
\newcommand{\Add}{\mathop{\rm Add}\nolimits} 
\newcommand{\aadd}{\mathop{\rm add}\nolimits} 
\newcommand\colim{\mathop{\rm colim}} 
\newcommand\hocolim{\mathop{\rm hocolim}} 

\newcommand\opp{{^{\text{\rm op}}}} 

\newcommand\CA{{\mathcal A}}

\newcommand\CC{{\mathcal C}}

\newcommand\FF{{\mathfrak F}} 

\newcommand\FG{{\mathfrak G}}

\newcommand\CP{{\mathcal P}}
\newcommand\CQ{{\mathcal Q}}

\newcommand\CS{{\mathcal S}}
\newcommand\CT{{\mathcal T}}

\newcommand\N{{\mathbb N}} 
\newcommand\Z{{\mathbb Z}} 

\begin{document}

\title[On perfectly generating projective classes]{On perfectly generating projective classes in triangulated categories}

\author{George Ciprian Modoi}
\thanks{The author was supported by the grant PN2CD-ID-489}

\address{"Babe\c s-Bolyai" University, Faculty of Mathematics and Computer Science, Chair of Algebra\\ 1, M. Kogalniceanu, RO-400084, Cluj-Napoca
\\ Romania}

\email[George Ciprian Modoi]{cmodoi@math.ubbcluj.ro}



\begin{abstract}
We say that a projective class in a triangulated category with
coproducts is perfect if the corresponding ideal is closed under
coproducts of maps. We study perfect projective classes and the
associated phantom and cellular towers. Given a perfect generating
projective class, we show that every object is isomorphic to the
homotopy colimit of a cellular tower associated to that object.
Using this result and the Neeman's Freyd--style representability
theorem we give a new proof of Brown Representability Theorem.
\end{abstract}
\maketitle

\section*{Introduction}

The notion of projective classes in pointed categories goes back
to Eilenberg and Moore \cite{EM}. In this paper we consider
projective classes in a category $\CT$ which is triangulated. In
this settings projective classes may be defined as pairs
$(\CP,\FF)$, with $\CP\subseteq\CT$ a class of objects and
$\FF\subseteq\CT^\to$ a class of maps (here $\CT^\to$ is the
category of all maps in $\CT$) such that $\CP$ is closed under
direct factors, $\FF$ is an ideal (that means $\phi,\phi\in\FF$,
and $\alpha,\beta\in\CT^\to$, implies
$\phi+\phi',\alpha\phi\beta\in\FF$, whenever the operations are
defined), the composite $p\to x\stackrel{\phi}\to x'$ is zero for
all $p\in\CP$ and all $\phi\in\FF$, and each object $x\in\CT$ lies
in an exact triangle $\Sigma^{-1}x'\to p\to x\stackrel{\phi}\to
x'$, with $p\in\CP$ and $\phi\in\FF$. Note also that all
projective classes which we deal with are stable under suspensions
and desuspensions in $\CT$. Fix an object $x\in\CT$. Choosing
repeatedly triangles as above, we construct two towers in $\CT$
associated to $x$, namely the phantom and the cellular tower. The
whole construction is similar to the choice of a projective
resolution for an object in an abelian category.

Let $\kappa$ be a regular cardinal. We say that a projective class
$(\CP,\FF)$ is $\kappa$--perfect, provided that the ideal $\FF$ is
closed under $\kappa$--coproducts in $\CT^\to$, that is coproducts
of less that $\kappa$ maps, respectively perfect if it is
$\kappa$-perfect for all cardinals $\kappa$. For projective
classes which are induced by sets our definition of perfectness is
equivalent to that of \cite{KL}, explaining our terminology.
Further we say that $(\CP,\FF)$ generates $\CT$ if for any
$x\in\CT$, $\CT(\CP,x)=0$ implies $x=0$. It seems that an
important role is played by $\aleph_1$--perfect projective
classes, that means projective classes $(\CP,\FF)$ with $\FF$
closed under countable coproducts. In this case we prove that the
homotopy colimit of a tower whose maps belong to $\FF$ is zero
(see Lemma \ref{hocolimis0}). In particular the homotopy colimit
of the phantom tower associated to an object vanishes. If, in
addition, we assume that $(\CP,\FF)$ generates $\CT$ then Theorem
\ref{xishocolim} tells us that every object $x$ is (isomorphic to)
the homotopy colimit of every associated cellular tower. Note also
that the hypothesis of $\aleph_1$--perfectness seems to be
implicitly assumed by Christensen in \cite{C}, as we may see from
Proposition \ref{pgen} and Remark \ref{chris}.

Using the product of two projective classes defined in \cite{C} we
recall the construction of the $n$-th power $(\CP^{*n},\FF^{*n})$
of a projective class $(\CP,\FF)$, for $n\in\N$. In \cite{NR} it
is shown that, if $(\CP,\FF)$ is induced by a set, then for every
cohomological functor $F:\CT\to\Ab$ which sends coproducts into
products the comma category $\CP^{*n}/F$ has a weak terminal
object, for all $n\in\N$. Provided that $(\CP,\FF)$ is
$\aleph_1$--perfect, we use the fact that every $x$ is the
homotopy colimit of its cellular tower in order to extend the
above property to the whole category $\CT/F$. We deduce a version
of Brown Representability Theorem for triangulated categories with
coproducts which are $\aleph_1$-perfectly generated by a
projective class satisfying the additional property that every
category $\CP^{*n}/F$ has a weak terminal object, for every
$n\in\N$ and every cohomological functor which sends coproducts
into products $F:\CT\to\Ab$ (see Theorem \ref{brtf}). In
particular if the projective class is induced by a set, then this
additional property is automatically fulfilled, and we obtain in
Corollary \ref{brt} the version of Brown Representability due to
Krause in \cite[Theorem A]{KB}, but our proof is completely
different, as it is based on the Freyd--style representability
theorem of \cite{NR}.

For our version of Brown Representability the finite powers of a
projective class is all what we need. We still treated the case of
transfinite ordinals, following a suggestion of Neeman (see
\cite[Remark 0.10]{NR}). A minor modification of the arguments in
\cite{NR} shows that if $\CT=\CP^{*i}$ for some ordinal $i$ then
the Brown representability theorem holds for $\CT$ . We fill in
the details this observation in Lemmas \ref{wtermpq} and
\ref{wtermpi}. On the other hand, if every $x\in\CT$ is the
homotopy colimit of its $\FF$--cellular tower, then
$\CT=\CP^{*\omega}*\CP^{*\omega}$, where $\omega$ is the first
infinite ordinal. But due to a technical detail we are not able to
deduce, as in the case of finite ordinals (see \cite[Note
3.6]{C}), that
$\CP^{*\omega}*\CP^{*\omega}=\CP^{*(\omega+\omega)}$.

In all categories we consider, the homomorphisms between two
objects form a set and not a genuine class. For undefined terms
and properties concerning triangulated categories we refer to
\cite{N}. The standard reference for abelian category is \cite{G}.
For general theory of categories we refer the reader to \cite{BM}
or \cite{PP}.

The author would like to thank to Henning Krause for his interest
in this work. He is also very indebted to an anonymous referee,
which pointed out a number of mistakes in a preliminary version of
this paper, leading to a radical change of the paper.

\section{Projective classes and associated towers}\label{pcat}

Consider a preadditive category $\CT$.  Then by a
{\em$\CT$-module\/} we understand a functor $X:\CT\opp\to\Ab$.
Such a functor is called {\em finitely presented\/} if there is an
exact sequence of functors \[\CT(-,y)\to\CT(-,x)\to X\to0\] for
some $x,y\in\CT$. Using Yoneda lemma, we know that the class of
all natural transformations between two $\CT$-modules $X$ and $Y$
denoted $\Hom_\CT(X,Y)$ is actually a set, provided that $X$ is
finitely presented. We consider the category $\md(\CT)$ of all
finitely presented $\CT$-modules, having as morphisms sets
$\Hom_\CT(X,Y)$ for all $X,Y\in\md(\CT)$. The Yoneda functor
\[H=H_\CT:\CT\to\md(\CT)\hbox{ given by }H_\CT(x)=\CT(-,x)\] is  an
embedding of $\CT$ into $\md(\CT)$, according to Yoneda lemma. If,
in addition, $\CT$ has coproducts then $\md(\CT)$ is cocomplete
and the Yoneda embedding preserves coproducts. It is also
well--known (and easy to prove) that, if $F:\CT\to\CA$ is a
functor into an additive category with cokernels, then there is a
unique, up to a natural isomorphism, right exact functor $\widehat
F:\md(\CT)\to\CA$, such that $F=\widehat F\circ H_\CT$ (see
\cite[Universal property 2.1]{KF}). Moreover, $F$ preserves
coproducts if and only if $\widehat F$ preserves colimits.

In this paper the category $\CT$ will be triangulated. Recall that
$\CT$ is supposed to be additive. A functor $\CT\to\CA$ into an
abelian category $\CA$ is called {\em homological\/} if it sends
triangles into exact sequences. A contravariant functor
$\CT\to\CA$ which is homological regarded as a functor
$\CT\opp\to\CA$ is called {\em cohomological\/} (see
\cite[Definition 1.1.7 and Remark 1.1.9]{N}). An example of a
homological functor is the Yoneda embedding
$H_\CT:\CT\to\md(\CT)$. We know: $\md(\CT)$ is an abelian
category, and  for every functor $F:\CT\to\CA$  into an abelian
category, the unique right exact functor $\widehat
F:\md(\CT)\to\CA$ extending $F$ is exact if and only if $F$ is
homological, by \cite[Lemma 2.1]{KS}.  Moreover, $\md(\CT)$ is a
Frobenius abelian category, with enough injectives and enough
projectives, by \cite[Corollary 5.1.23]{N}. Injective and
projective objects in $\md(\CT)$ are, up to isomorphism, exactly
objects of the form $\CT(-,x)$ for some $x\in\CT$, provided that
the idempotents in $\CT$ split.

From now on, we suppose $\CT$ has arbitrary coproducts, so the
idempotents in $\CT$ split according to \cite[Proposition
1.6.8]{N}. First we record some easy but useful results. Recall
that a {\em homotopy colimit} of a tower of objects and maps
\[x_0\stackrel{\phi_0}\to x_1\stackrel{\phi_1}\to
x_2\stackrel{\phi_2}\to x_3\to\cdots\] is defined via the triangle
\[\coprod_{n\in\N}x_n\stackrel{1-\phi}\rightarrow\coprod_{n\in\N}
x_n\rightarrow\hocolim x_n\to\Sigma\coprod_{n\in\N}x_n,\] where
$\phi$ is the unique morphism which makes commutative all the
diagrams of the form \[\diagram
x_n\rto\dto_{\phi_n}&\displaystyle{\coprod_{n\in\N}}x_n\dto^{\phi}\\
x_{n+1}\rto &\displaystyle{\coprod_{n\in\N}}x_n
\enddiagram\ (n\in\N).\]
Obviously, the homotopy colimit of a tower is unique, up to a non
unique isomorphism. We denote sometimes the map $\phi$ by $shift$,
especially if we don't need an explicit notation for the maps in
the tower.

The following Lemma is the dual of \cite[Lemma 5.8 (2)]{B}. Note
that we give a more general version, replacing the category $\Ab$
(more precisely $\Ab^\opp$) with an abelian AB4 category $\CA$,
where the derived functors $\colim^{(i)}$ of the colimits are
computed in the usual manner, by using homology of a complex.
Moreover, \cite[Lemma 5.8 (1)]{B} is a direct consequence of this
dual, together with the exactness of colimits in $\Ab$ (that is
$\colim^{(1)}=0$).

\begin{lemma}\label{fhocolim} Consider a tower $x_0\stackrel{\phi_0}\to x_1\stackrel{\phi_1}\to
x_2\stackrel{\phi_2}\to x_3\to\cdots$ in $\CT$. If $F:\CT\to\CA$ a
homological functor which preserves countable coproducts into an
abelian AB4 category $\CA$, then we have a Milnor exact sequence
\[0\to\colim F(x_n)\to F(\hocolim x_n)\to\colim{^{(1)}}F(\Sigma
x_n)\to0\] and $\colim^{(i)}F(x_n)=0$ for $i\geq2$.
\end{lemma}

\begin{corollary}\label{fhocol0} Consider a tower $x_0\stackrel{\phi_0}\to x_1\stackrel{\phi_1}\to
x_2\stackrel{\phi_2}\to x_3\to\cdots$ in $\CT$. If $F:\CT\to\CA$
is a homological functor, which preserves countable coproducts
into an abelian AB4 category, such that $F(\Sigma^i\phi_n)=0$ for
all $i\in\Z$ and all $n\geq0$, then $F(\hocolim x_n)=0$.
\end{corollary}

\begin{proof} With our hypothesis we have $\colim F(x_n)=0=\colim{^{(1)}}F(\Sigma x_n)$,
so $F(\hocolim x_n)=0$ by the
Milnor exact sequence of Lemma \ref{fhocolim}.
\end{proof}

Recall that a pair $(\CP,\FF)$ consisting of a class of objects
$\CP\subseteq\CT$ and a class of morphisms $\FF$ is called {\em
projective class\/} if $\Sigma^n(\CP)\subseteq\CP$ for all
$n\in\N$,
\[\CP=\{p\in\CT\mid \CT(p,\phi)=0\hbox{ for all }\phi\in\FF\},\]
\[\FF=\{\phi\in\CT\mid \CT(p,\phi)=0\hbox{ for all }p\in\CP\}\]
and each $x\in\CT$ lies in a triangle $\Sigma^{-1}x'\to p\to x\to
x'$, with $p\in\CP$ and $x\to x'$ in $\FF$ (see \cite{C}). Note
that we work only with projective classes which are stable under
(de)suspensions; generally it is possible to define a projective
class without this condition. Clearly, $\CP$ is closed under
coproducts and direct factors, $\FF$ is an ideal, and $\FF$ is
stable under (de)suspensions. Moreover $\CP$ and $\FF$ determine
each other. A triangle of the form $x\to y\to z\to\Sigma x$ is
called {\em $\FF$--exact\/} if the morphism $z\to\Sigma x$ belongs
to $\FF$. If this is the case, the morphisms $x\to y$ and $y\to z$
are called {\em $\FF$--monic\/}, respectively {\em$\FF$--epic\/}.

Let $(\CP,\FF)$ be a projective class in $\CT$. The inclusion
functor $\varphi:\CP\to\CT$ induces a unique right exact functor
$\varphi^*$ making commutative the diagram
\[\diagram
\CP\rto^{\varphi}\dto_{H_\CP}&\CT\dto^{H_\CT}\\
\md(\CP)\rto^{\varphi^*}&\md(\CT)\enddiagram\] where $H_\CP$ and
$H_\CT$ are the respective Yoneda functors. More explicitly,
\[\varphi^*(\CP(-,p))=\CT(-,p)\] for all $p\in\CP$, and $\varphi^*$ is right exact.
Moreover since $\varphi$ is fully--faithful, $\varphi^*$ has the
same property \cite[Lemma 2.6]{KF}.

A {\em weak kernel\/} for a morphism  $y\to z$ in a preadditive
category $\CC$ is a morphism $x\to y$ such that, the induced
sequence of abelian groups $\CC(t,x)\to\CC(t,y)\to\CC(t,z)$ is
exact for all $t\in\CC$. Return to the case of a projective class
$(\CP,\FF)$ in the triangulated category $\CT$.  To construct a
weak kernel of a morphism $q\to r$ in $\CP$ we proceed as follows:
The morphism fits into a triangle $x\to q\to r\to\Sigma x$; let
$\Sigma^{-1}x'\to p\to x\to x'$ an $\FF$--exact triangle with
$p\in\CP$; then the composite map $p\to x\to q$ gives the desired
weak kernel. Therefore $\md(\CP)$ is abelian (for example by
\cite[Lemma 2.2]{KF}, but this is also well--known). Moreover the
restriction functor
\[\varphi_*:\md(\CT)\to\md(\CP),\hbox{ }\varphi_*(X)=X\circ\varphi\hbox{ for all
}X\in\md(\CT)\] is well defined and it is the exact right adjoint
of $\varphi^*$, by \cite[Lemma 2]{KB}.

We know by \cite[Lemma 3.2]{C} that a pair $(\CP,\FF)$ is a
projective class, provided that $\CP$ is a class of objects closed
under direct factors, $\FF$ is an ideal, $\CP$ and $\FF$ are {\em
orthogonal\/} (that means, the composite $p\to x\to x'$ is zero
for all $p\in\CP$ and all $x\to x'$ in $\FF$) and each object
$x\in\CT$ lies in an $\FF$--exact triangle $\Sigma^{-1}x'\to p\to
x\to x'$, with $p\in\CP$. If $\CS$ is a set of objects in $\CT$,
then $\Add\CS$ denotes, as usual, the class of all direct factors
of arbitrary coproducts of objects in $\CS$. The following lemma
is straightforward (see also \cite[Definition 5.2 and the
following paragraph]{C}):

\begin{lemma}\label{projcl} Consider a set $\CS$ of objects
in $\CT$ which is closed under suspensions and desuspensions.
Denote by $\CP=\Add\CS$, and let $\FF$   be the class of all
morphisms $\phi$ in $\CT$ such that $\CT(s,\phi)=0$ for all
$s\in\CS$. Then $(\CP,\FF)$ is a projective class.
\end{lemma}

We will say that the projective class $(\CP,\FF)$ given in Lemma
\ref{projcl} is {\em induced\/} by the set $\CS$. Note also that
if $\CS$ is an essentially small subcategory of $\CT$, such that
$\Sigma^n(\CS)\subseteq\CS$ for all $n\in\Z$, then we will also
speak about the projective class induced by $\CS$, understanding
the projective class induced by a representative set of
isomorphism classes of objects in $\CS$. If, in particular,
$\kappa$ is a regular cardinal, $\CS$ consists of $\kappa$--small
objects and it is closed under coproducts of less than $\kappa$
objects (for example if $\CS$ is the subcategory of all
$\kappa$-compact object of $\CT$), then $\md(\CP)$ is equivalent
to the category of all functors $\CS\opp\to\Ab$ which preserve
products of less than $\kappa$ objects, by \cite[Lemma 2]{KN},
category used extensively in \cite{N} as a locally presentable
approximation of $\md(\CT)$.

\begin{remark}\label{epic} Under the hypotheses of Lemma \ref{projcl},
a map $x\to y$ in $\CT$ is $\FF$--monic  ($\FF$--epic) if and only
if the induced map $\CT(s,x)\to\CT(s,y)$ injective (respectively
surjective) for all $s\in\CS$.
\end{remark}

As in \cite{B} and \cite{C}, given a projective class $(\CP,\FF)$
in $\CT$, we construct two towers of morphisms associated to each
$x\in\CT$ as follows: We denote $x_0=\Sigma^{-1}x$. Inductively,
if $x_n\in\CT$ is given, for $n\in\N$, then there is an
$\FF$-exact triangle
\[\Sigma^{-1}x_{n+1}\to p_n\to x_n\stackrel{\phi_n}\to x_{n+1}\] in
$\CT$, by definition of a projective class. Consider then the
tower:
\[\Sigma^{-1}x=x_0\stackrel{\phi_0}\to x_1\stackrel{\phi_1}\to
x_2\stackrel{\phi_2}\to x_3\to\cdots.\] Such a tower is called a
{\em$\FF$-phantom tower} of $x$. The explanation of the
terminology is that morphisms $\phi_n$ belong to $\FF$ for all
$n\in\N$, and $\FF$ may be thought as a generalization of the
ideal of classical phantom maps in a compactly generated
triangulated category. (Clearly, $\FF$ coincides with the ideal of
classical phantom maps, provided that the projective class
$(\CP,\FF)$ is induced by the full essentially small subcategory
consisting of all compact objects.)

Observe that there are more $\FF$-phantom towers associated to the
same element $x\in\CT$, according with the choices of the
$\FF$-epic map $p_n\to x_n$ at each step $n\in\N$. The analogy
with projective resolutions in abelian categories is obvious.

Choose an $\FF$-phantom tower of $x\in\CT$ as in the definition
above. We denote by $\phi^n$ the composed map
$\phi_{n-1}\dots\phi_1\phi_0:\Sigma^{-1}x\to x_n$, for all
$n\in\N^*$, and we set $\phi^0=1_{\Sigma^{-1}x}$. Then let $x^n$
be defined, uniquely up to a non unique isomorphism, by the
triangle $\Sigma^{-1}x\stackrel{\phi^n}\to x_n\to x^n\to x$. The
octahedral axiom allows us to complete the commutative diagram
\[\diagram &p_n\rdouble\dto&p_n\dto&\\
\Sigma^{-1}x\rto^{\phi^n}\ddouble& x_n\rto\dto^{\phi_n}& x^n\rto\dto&x\ddouble\\
\Sigma^{-1}x\rto^{\phi^{n+1}}& x_{n+1}\rto\dto& x^{n+1}\rto\dto &x\\
&\Sigma p_n\rdouble&\Sigma p_n&\\
\enddiagram\] with the triangle in the second column.

Therefore we obtain an another tower of objects \[0=x^0\to x^1\to
x^2\to x^3\to\cdots,\] where for each $n\in\N$ we have a triangle
$p_n\to x^n\to x^{n+1}\to\Sigma p_n$, with $p_n\in\CP$ chosen in
the construction of the above $\FF$-phantom tower. Such a tower is
called a {\em $\FF$-cellular tower} of $x\in\CT$.

Considering homotopy colimits of the $\FF$-phantom and
$\FF$-cellular towers, we obtain a sequence
\[\Sigma^{-1}x\to\hocolim x_n\to\hocolim x^n\to x.\] It is not known
whether the induced sequence can be chosen to be a triangle (see
\cite[p. 302]{B}). However the answer to this question is yes,
provided that $\CT$ is the homotopy category of a suitable stable
closed model category in the sense of \cite{H}.

\begin{proposition}\label{exseq} Let $(\CP,\FF)$ be a projective class in
$\CT$, and let denote by $\varphi:\CP\to\CT$ the inclusion
functor. For every $x\in\CT$ we consider an $\FF$-phantom tower
and an $\FF$-cellular tower as above. Then we have an exact
sequence
\[0\to\coprod(\varphi_*\circ H_\CT)(x^n)\stackrel{1-shift}\longrightarrow\coprod(\varphi_*\circ H_\CT)(x^n)
\to(\varphi_*\circ H_\CT)(x)\to0,\] where
$\varphi_*:\md(\CT)\to\md(\CP)$ is the restriction functor.
Consequently
\[\colim(\varphi_*\circ H_\CT)(x^n)\cong(\varphi_*\circ H_\CT)(x)\hbox{
and }\colim{^{(1)}}(\varphi_*\circ H_\CT)(x^n)=0.\]
\end{proposition}

\begin{proof}
By applying the functor $\varphi_*\circ H_\CT$ to the diagram
above defining an $\FF$-cellular tower associated to $x$, we
obtain a commutative diagram in the abelian category with
coproducts $\md(\CP)$:
\[\diagram 0\rto&(\varphi_*\circ H_\CT)(x_n)\rto\dto_{0}&(\varphi_*\circ H_\CT)(x^n)\rto\dto&(\varphi_*\circ
H_\CT)(x)\rto\ddouble&0\\
0\rto&(\varphi_*\circ H_\CT)(x_{n+1})\rto&(\varphi_*\circ
H_\CT)(x^{n+1})\rto&(\varphi_*\circ H_\CT)(x)\rto&0
\enddiagram\] The conclusion follows by \cite[Lemma 7.1.2]{KL}.
\end{proof}

\section{Perfectly generating projective classes}\label{pgpc}

Consider a cardinal $\kappa$. Recall that $\kappa$ is said to be
{\em regular\/} provided that it is infinite and it can not be
written as a sum of less than $\kappa$ cardinals, all smaller than
$\kappa$. By {\em$\kappa$-coproducts} we understand coproducts of
less that $\kappa$-objects.

\begin{proposition}\label{fcop} Let $\kappa$ be a regular cardinal and
let $(\CP,\FF)$ be a projective class in $\CT$. Denote by
$\varphi:\CP\to\CT$ the inclusion functor. Then the functor
$\varphi_*:\md(\CT)\to\md(\CP)$, $\varphi_*(X)=X\circ\varphi$
preserves $\kappa$--coproducts if and only if $\FF$ is closed
under $\kappa$--coproducts (of maps).
\end{proposition}

\begin{proof} The exact functor $\varphi_*$ having a fully--faithful
left adjoint induces an equivalence
$\md(\CT)/\Ker\varphi_*\to\md(\CP)$. Since $\md(\CT)$ is AB4, we
know that $\varphi_*$ preserves $\kappa$--coproducts if and only
if $\Ker\varphi_*$ is closed under $\kappa$--coproducts. Obviously
$\FF=\{\phi\mid(\varphi_*\circ H_\CT)(\phi)=0\}$.  Using the proof
of \cite[Section 3]{KS}, we observe that
\[\Ker\varphi_*=\{X\in\md(\CT)\mid X\cong\ima H_\CT(\phi)\hbox{ for some }\phi\in\FF\}.\]
Now suppose $\FF$ to be closed under $\kappa$--coproducts, and let
$\{M_\lambda\mid\lambda\in\Lambda\}$ be a set of objects in
$\Ker\varphi_*$, with the cardinality less than $\kappa$. Thus
$M_\lambda\cong\ima H_\CT(\phi_\lambda)$ for some
$\phi_\lambda\in\FF$, for all $\lambda\in\Lambda$. Therefore,
using again condition AB4 (coproducts in $\md(\CT)$ are exact, so
they commute with images), we obtain:  \[
\coprod_{\lambda\in\Lambda}M_\lambda\cong\coprod_{\lambda\in\Lambda}\ima
H_\CT(\phi_\lambda)\cong
\ima\left(\coprod_{\lambda\in\Lambda}H_\CT(\phi_\lambda)\right)\cong
\ima H_\CT\left(\coprod_{\lambda\in\Lambda}\phi_\lambda\right),\]
  showing that $\coprod_{\lambda\in\Lambda}M_\lambda\in\Ker\varphi_*$.

Conversely, if $\Ker\varphi_*$ is closed under
$\kappa$--coproducts, and $\{\phi_\lambda\mid\lambda\in\Lambda\}$
is a set of maps in $\FF$, with the cardinality less than
$\kappa$, then
\[\varphi_*\left(\ima
H_\CT\left(\coprod_{\lambda\in\Lambda}\phi_\lambda\right)\right)=
  \varphi_*\left(\coprod_{\lambda\in\Lambda}\ima H_\CT(\phi_\lambda)\right)=0,\]
so $\FF$ is closed under $\kappa$--coproducts.
\end{proof}

We call {\em $\kappa$--perfect} the projective class $(\CP,\FF)$
if the equivalent conditions of Proposition \ref{fcop} hold true.
The projective class will be called {\em perfect} if it is
$\kappa$--perfect for all regular cardinals $\kappa$, that is,
$\FF$ is closed under arbitrary coproducts. Following \cite{C}, we
say that a projective class $(\CP,\FF)$ {\em generates\/} $\CT$ if
for any $x\in\CT$, we have $x=0$ provided that $\CT(p,x)=0$, for
each $p\in\CP$. Immediately, we can see that $(\CP,\FF)$ generates
$\CT$ if and only if $\varphi\circ H_\CT:\CT\to\md(\CP)$ reflects
isomorphisms, that is, if $\alpha:x\to y$ is a morphism in $\CT$
such that the induced morphism $(\varphi\circ H_\CT)(\alpha)$ is
an isomorphism in $\md(\CP)$, then $\alpha$ is an isomorphism in
$\CT$, where $\varphi:\CP\to\CT$ denotes, as usual, the inclusion
functor. Another equivalent statement is $\FF$ does not contain
non--zero identity maps. Consider now an essentially small
subcategory $\CS$ of $\CT$ which is closed under suspensions and
desuspensions, and $(\CP,\FF)$ the projective class induced by
$\CS$. Since coproducts of triangles are triangles, we conclude by
Remark \ref{epic} that $\FF$ is closed under coproducts exactly if
$\CS$ satisfies the following condition: If $x_i\to y_i$ with
$i\in I$ is a family of maps, such that $\CT(s,x_i)\to\CT(s,y_i)$
is surjective for all $i\in I$, then the induced map
$\CT(s,\coprod x_i)\to\CT(s,\coprod y_i)$ is also surjective. Thus
$(\CP,\FF)$ perfectly generates $\CT$ in the sense above if and
only if $\CS$ perfectly generates $\CT$ in the sense given in
\cite[Section 5]{KL} (see also \cite{KB} for a version relativized
at the cardinal $\kappa=\aleph_1$).

\begin{lemma}\label{hocolimis0} Consider a tower $x_0\stackrel{\phi_0}\to x_1\stackrel{\phi_1}\to
x_2\stackrel{\phi_2}\to x_3\to\cdots$ in $\CT$. If $(\CP,\FF)$ is
an $\aleph_1$--perfect projective class in $\CT$ and
$\phi_n\in\FF$ for all $n\geq0$, then $\hocolim x_n=0$.
\end{lemma}

\begin{proof} We apply Corollary \ref{fhocol0} to the
homological functor, which preserves countable coproducts
$\varphi_*\circ H_T:\CT\to\md(\CP)$, where $\varphi:\CP\to\CT$ is
the inclusion functor.
\end{proof}

\begin{proposition}\label{pgen}
If $(\CP,\FF)$ is a $\aleph_1$--perfect projective class in $\CT$,
then a necessary and sufficient condition for $(\CP,\FF)$ to
generate $\CT$ is
\[\lim_{n\in\N}\CT(x_n,y)=0=\lim_{n\in\N}{^{(1)}}\CT(x_n,y),\] for
all $x,y\in\CT$ and any choice \[x=x_0\stackrel{\phi_0}\to
x_1\stackrel{\phi_1}\to x_2\stackrel{\phi_2}\to x_3\to\cdots,\] of
an $\FF$-phantom tower of $x$. Here by $\lim^{(1)}$ we understand
the first derived functor of the limit.
\end{proposition}

\begin{proof} Let show the sufficiency first. If $x\in\CT$ has the
  property $\CT(p,x)=0$ for all $p\in\CP$, then $1_x\in\FF$ and  a
  $\FF$-phantom tower of $x$ is
\[x=x_0\stackrel{1_x}\to x_1=x\stackrel{1_x}\to x_2=x\to\cdots.\]
Then $0=\displaystyle{\lim_{n\in\N}}\CT(x_n,x)=\CT(x,x)$, so
$x=0$.

Now we show the necessity. Let $x,y\in\CT$ and consider an
$\FF$-phantom tower of
  $x$ as above. Applying the functor $\CT(-,y)$ to this tower, we
  obtain a sequence of abelian groups:
  \[\CT(x,y)=\CT(x_0,y)\stackrel{(\phi_0)_*}\leftarrow\CT(x_1,y)
  \stackrel{(\phi_1)_*}\leftarrow\CT(x_2,y)\stackrel{(\phi_2)_*}\leftarrow\CT(x_3,y)
\leftarrow\cdots.\] Computing the derived functors of the
  limit of such a sequence in the usual manner, we know that
  $\lim^{(n)}$ is zero for $n\geq2$ and $\lim$, $\lim^{(1)}$ are
given by the exact sequence:
  \[0\to\lim_{n\in\N}\CT(x_n,y)\to\prod_{n\in\N}\CT(x_n,y)
\stackrel{(1-\phi)_*}\to\prod_{n\in\N}\CT(x_n,y)\to\lim_{n\in\N}{^{(1)}}\CT(x_n,y)\to0,\]
where $\phi:\coprod_{n\in\N}x_n\to\coprod_{n\in\N}x_n$ is
constructed as in the definition of the homotopy colimit. Applying
$\CT(p,-)$  to the commutative squares which define $\phi$, we
obtain also commutative squares: \[\diagram
\CT(p,x_n)\rto\dto_{0=\CT(p,\phi_n)}&
\CT(p,\displaystyle{\coprod_{n\in\N}}x_n)\dto^{\CT(p,\phi)}\\
\CT(p,x_{n+1})\rto &\CT(p,\displaystyle{\coprod_{n\in\N}}x_n)
\enddiagram\ (n\in\N),\] for all $p\in\CP$.  According to
Proposition \ref{fcop}, the $\aleph_1$-perfectness of $(\CP,\FF)$
means that $\CT(-,\coprod_{n\in\N}x_n)|_\CP$ is the coproduct in
$\md(\CP)$ of the set \[\{\CT(-,x_n)|_\CP\mid n\in\N\},\] thus we
deduce $\CT(p,\phi)=0$. Now
$\CT(p,1-\phi)=\CT(p,1)-\CT(p,\phi)=\CT(p,1)$ is an isomorphism,
for all $p\in\CP$, so $1-\phi$ is an isomorphism, because
$(\CP,\FF)$ generates $\CT$. Consequently
\[\lim_{n\in\N}\CT(x_n,y)=0=\lim_{n\in\N}{^{(1)}}\CT(x_n,y).\]
\end{proof}

\begin{remark}\label{chris} {The hypotheses of
Proposition \ref{pgen} are almost identical with those of
\cite[Proposition 4.4]{C}, except the fact that we require, in
addition, the $\aleph_1$-perfectness for $(\CP,\FF)$. Moreover,
the conclusion of \cite[Proposition 4.4]{C} (namely: the Adams
spectral sequence abutting $\CT(x,y)$ is conditionally convergent)
is equivalent to our conclusion ($\lim$ and $\lim^{(1)}$ to be
zero). The proofs are also almost identical. Despite that, we have
given a detailed proof, because, without our additional condition,
we do not see how we can conclude, with our notations, that
$\CT(p,\phi)=0$. Thus we fill a gap existing in the proof of
\cite[Proposition 4.4]{C}, due to the missing assumption of
$\aleph_1$-perfectness. On the other hand, we do not have a
counterexample showing that the conclusion cannot be inferred
without this assumption, so the problem is open. Note also that
the terms of the Adams spectral sequence of \cite{C} do not
depend, for sufficiently large indices, of the choice of the
$\FF$-projective resolution of $x\in\CT$, so the conclusion of
Proposition \ref{pgen} may be formulated simply: The Adams
spectral sequence abutting $\CT(x,y)$ is conditionally convergent,
for any two $x,y\in\CT$.}
\end{remark}

\begin{theorem}\label{xishocolim}
Let $(\CP,\FF)$ be an $\aleph_1$--perfectly generating projective
class in $\CT$. Then for every $x\in\CT$, and every choice
\[0=x^0\to x^1\to x^2\to x^3\to\cdots\] of an $\FF$-cellular tower
for $x$ we have $\hocolim x^n\cong x$.
\end{theorem}

\begin{proof}
The homotopy colimit of the $\FF$-cellular tower above is
constructed via triangle
\[\coprod_{n\in\N}x^n\stackrel{1-shift}\longrightarrow\coprod_{n\in\N}
x^n\rightarrow\hocolim x^n\to\Sigma\coprod_{n\in\N}x^n.\] We apply
to this triangle the homological functor $\varphi_*\circ H_\CT$
which commutes with countable coproducts. Comparing the resulting
exact sequence with the exact sequence given by Proposition
\ref{exseq}, we obtain a unique isomorphism \[(\varphi_*\circ
H_\CT)(\hocolim x^n)\to(\varphi_*\circ H_\CT)(x),\] which must be
induced by the map $\hocolim x^n\to x$. The generating hypothesis
tells us that $\hocolim x^n\cong x$.
\end{proof}

Recall that $\aleph_1$--localizing subcategory of $\CT$ means
triangulated and closed under countable coproducts.

\begin{corollary}\label{aleph1}
If $(\CP,\FF)$ is an $\aleph_1$--perfectly generating projective
class in $\CT$, then $\CT$ is the smallest $\aleph_1$--localizing
subcategory of $\CT$, which contains $\CP$.
\end{corollary}

\begin{proof}
Let $x\in\CT$ and let \[0=x^0\to x^1\to x^2\to x^3\to\cdots\] be
an $\FF$-cellular tower for $x$. Since for every $n\geq0$ there
exits a triangle $p_n\to x_n\to x_{n+1}\to\Sigma p_n$, with
$p_n\in\CP$ (see the definition of an $\FF$-cellular tower), we
may see inductively that $x_n$ belongs to the smallest
triangulated subcategory of $\CT$ which contains $\CP$. Now
$\hocolim x^n$ belongs to the smallest $\aleph_1$-localizing
subcategory of $\CT$ which contains $\CP$, and the conclusion
follows by Theorem \ref{xishocolim}.
\end{proof}

\begin{remark}\label{hocolex}
Let $(\CP,\FF)$ be an $\aleph_1$--perfectly generating projective
class in $\CT$, and $x\in\CT$. If we chose an $\FF$-phantom tower
\[\Sigma^{-1}x=x_0\stackrel{\phi_0}\to x_1\stackrel{\phi_1}\to
x_2\stackrel{\phi_2}\to x_3\to\cdots\] and an $\FF$-cellular tower
\[0=x^0\to x^1\to x^2\to x^3\to\cdots\] for $x$, then $\hocolim
x_n=0$ by Lemma \ref{hocolimis0}, and $\hocolim x^n\cong x$ by
Theorem \ref{xishocolim}. Thus the triangle
$\Sigma^{-1}x\to\hocolim x_n\to\hocolim x^n\to x$ is trivially
exact.
\end{remark}

\begin{remark}
A filtration analogous to that of Theorem \ref{xishocolim}, for
the case of well--generated triangulated categories may be found
in \cite[Lemma B 1.3]{N}.
\end{remark}

\section{Brown representability via perfect projective
classes}\label{brpc}

For two projective classes $(\CP,\FF)$ and $(\CQ,\FG)$, we define
the {\em product} by
\begin{align*}\CP*\CQ=\aadd\{x\in\CT\mid&\hbox{there is a triangle }\\
&q\to x\to p\to\Sigma q\hbox{ with }p\in\CP,q\in\CQ\},\end{align*}
and $\FF*\FG=\{\phi\psi\mid\phi\in\FF,\psi\in\FG\}$. Generally by
$\aadd$ we understand the closure under finite coproducts and
direct factors. Since in our case the closure under arbitrary
coproducts is automatically fulfilled, $\aadd$ means here simply
the closure under direct factors. Thus $(\CP*\CQ,\FF*\FG)$ is a
projective class, by \cite[Proposition 3.3]{C}.

If $(\CP_i,\FF_i)$ for $i\in I$ is a family of projective classes,
then
\[\left(\Add\left(\bigcup_I\CP_i\right),\bigcap_I\FF_i\right)\] is
also a projective class by \cite[Proposition 3.1]{C}, called the
{\em meet} of the above family.

In a straightforward manner we may use the octahedral axiom in
order to show that the product defined above is associative. We
may also observe without difficulties that the product of two
(respectively the meet of a family of) $\kappa$-perfect projective
classes is also $\kappa$-perfect, where $\kappa$ is an arbitrary
regular cardinal.

Consider now a projective class $(\CP,\FF)$ in $\CT$. We define
inductively $\CP^{*0}=\{0\}$, $\FF^{*0}=\CT^\to$ and
$\CP^{*i}=\CP*\CP^{*(i-1)}$, $\FF^{*i}=\FF*\FF^{*(i-1)}$, for
every non--limit ordinal $i>0$. If $i$ is a limit ordinal then
$(\CP^{*i},\FF^{*i})$ is defined as the meet of all
$(\CP^{*j},\FF^{*j})$ with $j<i$. Therefore $(\CP^{*i},\FF^{*i})$
is a projective class for every ordinal $i$, which is called the
{\em $i$-th power} of the projective class of $(\CP,\FF)$ (see
also \cite{C}, for the case of ordinals less or equal to the first
infinite ordinal). Clearly we have $\CP^{*j}\subseteq\CP^{*i}$,
for all ordinals $j\leq i$.

\begin{remark}\label{xninpn} We can inductively see that for $x\in\CT$
it holds $x^n\in\CP^{*n}$ for all $n\in\N$, where $x^n$ is the
$n$--th term of an $\FF$-cellular tower of $x$.
\end{remark}

For example, if $\CT$ is compactly generated, and $\CT^c$ denotes
the subcategory of all compact objects, then the projective class
induced by $\CT^c$ is obviously perfect, thus we obtain immediate
consequence of Theorem \ref{xishocolim}:

\begin{corollary}\label{abelig}\cite[Corollary 6.9]{B}
If $\CT$ is compactly generated then any object $x\in\CT$ is the
homotopy colimit $\hocolim x^n$ of a tower $x^0\to x^1\to\cdots$,
where $x^n\in\Add(\CT^c)^{*n}$, for all $n\in\N$.
\end{corollary}

Consider a contravariant functor $F:\CT\to\Ab$. For a full
subcategory $\CC$ of $\CT$, we consider the comma category $\CC/F$
with the objects being pairs of the form $(x,a)$, where $x\in\CC$
and $a\in F(x)$, and maps
\[(\CC/F)((x,a)(y,b))=\{\alpha\in\CT(x,y)\mid F(\alpha)(b)=a\}.\]
Motivated by \cite{NR} it is interesting to find weak terminal
objects in $\CT/F$, that is objects $(t,b)\in\CT/F$, such that for
every $(x,a)\in\CT/F$ there is a map
$(x,a)\to(t,b)\in(\CT/F)^\to$. Another equivalent formulation of
this fact is that the natural transformation $\CT(-,t)\to F$ which
corresponds under the Yoneda isomorphism to $b\in F(t)$ is an
epimorphism. The statement a) of the following lemma is proved by
the same argument as \cite[Lemma 2.3]{NR}. We include a sketch of
the proof for the convenience of the reader.

\begin{lemma}\label{wtermpq} Let $F:\CT\to\Ab$ be a cohomological functor which sends coproducts into products.
\begin{itemize}
\item[{\rm a)}] If $(\CP,\FF)$ and $(\CQ,\FG)$ are projective
classes in $\CT$ such that $(\CP,\FF)$ is induced by a set and
$\CQ/F$ has a weak terminal object, then $(\CP*\CQ)/F$ has a weak
terminal object. \item[{\rm b)}] If $(\CP_i,\FF_i)$, $i\in I$ are
projective classes in $\CT$ with the meet $(\CP,\FF)$, and
$\CP_i/F$ has a weak terminal object for all $i\in I$ then $\CP/F$
has a weak terminal object.
\end{itemize}
\end{lemma}

\begin{proof} a) Let $(q,d)$ be a weak terminal object in
$\CQ/F$, and let $\CS$ be a set which induces the projective class
$(\CP,\FF)$.  Obviously $\CP/F$ has a weak terminal object
$(p,c)$. Consider an object $(y,a)\in(\CP*\CQ)/F$. Thus there is a
triangle $x\to y\to z\to \Sigma x$, with $x\in\CQ$ and $z\in\CP$.
We have $z\amalg z'=\coprod_{i\in I}s_i$ for some $z'\in\CT$. We
construct the commutative diagram in $\CT$ whose rows are
triangles:
\[\diagram
x\rto\ddouble&y\rto\dto&z\rto\dto&\Sigma x\ddouble\\
x\rto^{\alpha}\dto_{f}&y\amalg z'\rto\dto^{g}&\coprod_{i\in I}s_i\rto\ddouble&\Sigma x\dto\\
q\rto_{\beta}&y_1\rto&\coprod_{i\in I}s_i\rto_{\gamma}&\Sigma
q\enddiagram\] We proceeded as follows: The triangle on the second
row is obtained as the coproduct of the initial one with $0\to
z'\to z'\to 0$, and the maps are the canonical injections. For
$d'=F(\alpha)(a,0)\in F(x)$, there is a map
$f:(x,d')\to(q,d)\in(\CQ/F)^\to$, since $(q,d)$ is weak terminal.
The first bottom square of the diagram above is homotopy push--out
(see \cite[Definition 1.4.1 and Lemma 1.4.4]{N}). Clearly
$y_1\in\CP*\CQ$. Since $F$ is cohomological, there is $a_1\in
F(y_1)$ such that $F(\beta)(a_1)=d$ and $F(g)(a_1)=(a,0)$. So if
we find a map $(y_1,a_1)\to(t,b)\in((\CP*\CQ)/F)^\to$ for a fixed
object $(t,b)$, then the conclusion follows.

If we denote by $J\subseteq\bigcup_{s\in\CS}\CT(s,\Sigma q)$ the
set of all maps $s_i\to\coprod_{i\in I}s_i\to\Sigma q$, then
$\gamma$ factors as $\coprod_{i\in
I}s_i\stackrel{\nabla}\longrightarrow\coprod_{s\in J}s\to\Sigma
q$, where $\nabla$ is a split epimorphism. Hence the fibre of
$\gamma$ is isomorphic to $y_J\amalg z''$, for some $z''\in\CP$
and $y_J$ defined as the fibre of the canonical map $\coprod_{s\in
J}s\to\Sigma q$. Therefore $(y,a)$ maps to $(t,b)=(t'\amalg
p,(b',c))$ where
\[(t',b')=\left(\coprod_{J\subseteq\bigcup_{s\in\CS}\CT(s,\Sigma q)}\left(\coprod_{u\in F(y_J)}(y_J,u)\right)\right),\]
so the object $(t,b)$ is weak terminal in $(\CP*\CQ)/F$.

b) If $(t_i,a_i)\in\CP_i/F$ is a weak terminal object, then
$(\coprod_{i\in I}t_i,(a_i)_{i\in I})$ is a weak terminal object
in $\CP/F$.
\end{proof}

By transfinite induction we obtain:

\begin{lemma}\label{wtermpi}
Let $(\CP,\FF)$ be a projective class in $\CT$ which is induced by
a set. For every ordinal $i$ and every cohomological functor
$F:\CT\to\Ab$ which sends coproducts into products, the category
$\CP^{*i}/F$ has a weak terminal object.
\end{lemma}

\begin{remark}\label{wtermn} For finite ordinals, Lemma \ref{wtermpi} is the same as
\cite[Lemma 2.3]{NR}. Note also that Neeman defined the operation
$*$ without to assume the closure under direct factors, but for a
subcategory $\CC$ of $\CT$ such that $(t,b)$ is weak terminal in
$\CC/F$, the same object is weak terminal in $\aadd\CC/F$ too.
\end{remark}

\begin{proposition}\label{wterm}
Let $(\CP,\FF)$ be an $\aleph_1$--perfectly generating projective
class in $\CT$, and let $F:\CT\to\Ab$ be a cohomological functor
which sends coproducts into products. Suppose also that every
category $\CP^{*n}/F$ has a weak terminal object $(t^n,b_n)$, for
$n\in\N$. Then $\CT/F$ has a weak terminal object.
\end{proposition}

\begin{proof}
Denote by $I$ the set of all towers
$0=t^0\stackrel{\tau_0}\longrightarrow
t^1\stackrel{\tau_1}\longrightarrow t^2\to\cdots$, satisfying
$F(\tau_n)(b_{n+1})=b_n$, for all $n\in\N$. The set $I$ is not
empty since for all $n\in\N$, we have
$t^n\in\CP^{*n}\subseteq\CP^{*(n+1)}$ and $(t^{n+1},b_{n+1})$ is
weak terminal in $\CP^{*(n+1)}/F$. Denote also by $t_i$ the
homotopy colimits of the tower $i\in I$, and chose $b_i\in F(t_i)$
an element which maps into $(b_n)_{n\in\N}$ via the surjective
(see the dual of Lemma \ref{fhocolim}) map
$F(t_i)\to\displaystyle{\lim_{n\in\N}}F(t^n)$. We claim that
\[(t,b)=\left(\coprod_{i\in I}t_i,(b_i)_{i\in I}\right)\in\CT/F\] is a weak
terminal object.

In order to prove our claim, let $x\in\CT$. As we have seen in
Theorem \ref{xishocolim}, it is isomorphic to the the homotopy
colimit of its $\FF$-cellular tower
$0=x^0\stackrel{\alpha_0}\longrightarrow
x^1\stackrel{\alpha_1}\longrightarrow x^2\to\cdots$, associated
with a choice of an $\FF$-phantom tower. Thus consider the
commutative diagram, whose rows are exact by Lemma \ref{fhocolim}
and whose vertical arrows are induced by the natural
transformation corresponding to $b\in F(t)$ via the Yoneda
isomorphism:
\[\diagram 0\rto&\lim^{(1)}\CT(\Sigma x^n,t)\rto\dto&\CT(x,t)\rto\dto&\lim\CT(x^n,t)\rto\dto&0\\
0\rto&\lim^{(1)}F(\Sigma x^n)\rto&F(x)\rto&\lim F(x^n)\rto&0
\enddiagram\] If we would prove that the two extreme vertical
arrows are surjective, then the middle arrow enjoys the same
property and our work would be done.

For $n\in\N$, we know that $\Sigma x^n\in\CP^{*n}$ and $(t^n,b_n)$
is weak terminal in $\CP^{*n}$, so there is a map $(\Sigma
x^n,a_n)\to(t^n,b_n)\in(\CP^{*n}/F)^\to$ for every element $a_n\in
F(\Sigma x^n)$. Because $I\neq\emptyset$, there exists $i\in I$,
hence we obtain a map
\[(\Sigma x^n,a_n)\to(t^n,b_n)\to(t_i,b_i)\to(t,b)\in(\CT/F)^\to\]
showing that the natural map $\CT(\Sigma x^n,t)\to F(\Sigma x^n)$
is surjective. Therefore the first vertical map in the commutative
diagram above is surjective as we may see from the following
commutative diagram with exact rows:
\[\diagram \prod\CT(\Sigma x^n,t)\rto^{1-shift}\dto&\prod\CT(\Sigma x^n,t)\rto\dto&\lim^{(1)}\CT(\Sigma x^n,t)\rto\dto&0\\
\prod F(\Sigma x^n)\rto^{1-shift}&\prod F(\Sigma
x^n)\rto&\lim^{(1)}F(\Sigma x^n)\rto&0
\enddiagram\]

Let show now that the map $\lim\CT(x^n,t)\to\lim F(x^n)$ is
surjective too. Consider an element $(a_n)\in\lim F(x^n)$, that is
$a_n\in F(x^n)$ such that $a_n=F(\alpha_n)(a_{n+1})$ for all
$n\in\N$. We want to construct a commutative diagram \[\diagram
x^0\rto^{\alpha_0}\dto^{f^0}&x^1\rto^{\alpha_1}\dto^{f^1}&x^2\rto\dto^{f^2}&\cdots\\
t^0\rto_{\tau_0}&t^1\rto_{\tau_1}&t^2\rto&\cdots
\enddiagram\] such that the bottom row is a tower in $I$ and
$F(f^n)(b^n)=a_n$ for all $n\in\N$. We proceed inductively as
follows: $f^0=0$ and $f^1$ comes from the fact that $(t^1,b_1)$ is
weak terminal in $\CP/F$. Suppose that the construction is done
for the first $n$ steps. Further we construct a commutative
diagram in $\CT$, where the rows are triangles and the second
square is homotopy push--out (see \cite[Definition 1.4.1 and Lemma
1.4.4]{N}):
\[\diagram
p_n\rto\ddouble& x^n\rto^{\alpha_n}\dto_{f^n}&x^{n+1}\rto\dto&\Sigma p_n\ddouble\\
p_n\rto&t^n\rto&y^{n+1}\rto&\Sigma p_n\enddiagram\] By
construction $p_n\in\CP$, hence $y^{n+1}\in\CP^{*(n+1)}$. On the
other hand $y^{n+1}$ is obtained via the triangle
\[x^n\stackrel{\left(\begin{array}{c}\alpha_n\\-f^n\end{array}\right)}\longrightarrow x^{n+1}\amalg t^n\to y^{n+1}\to\Sigma x^n,\]
therefore the sequence
\[F(y^{n+1})\to F(x^{n+1})\times F(t^n)\stackrel{\left(F(\alpha_n),-F(f^n)\right)}\longrightarrow F(x^n)\]
is exact in $\Ab$. Because
$F(\alpha_n)(a_{n+1})-F(f^n)(b_n)=a_n-a_n=0$, we obtain an element
$b_{n+1}'\in F(y^{n+1})$ which is sent to $(a_{n+1},b_n)$ by the
first map in the exact sequence above. Thus the two maps
constructed in the homotopy push-out square above are actually
maps $(x^{n+1},a_{n+1})\to(y^{n+1},b_{n+1}')$ respectively
$(t^n,b_n)\to(y^{n+1},b_{n+1}')$ in $\CP^{*(n+1)}/F$. Since
$(t^{n+1},b_{n+1})$ is weak terminal in $\CP^{*(n+1)}/F$, they can
be composed with a map
$(y^{n+1},b_{n+1}')\to(t^{n+1},b_{n+1})\in(\CP^{*(n+1)}/F)^\to$,
in order to obtain a commutative square \[\diagram
x^n\rto^{\alpha_n}\dto_{f^n}&x^{n+1}\dto^{f^{n+1}}\\
t^n\rto_{\tau_n}&t^{n+1}
\enddiagram\] as desired. Denote by $i\in I$ the tower constructed above.
We have a composed map $F(t)\to F(t_i)\to\lim F(t^n)\to\lim
F(x^n)$ which sends $b\in F(t)$ in turn into $b_i$, then into
$(b_n)_{n\in\N}$ and finally into $(a_n)_{n\in\N}$. This shows
that the element $(a_n)_{n\in\N}\in\lim F(x^n)\subseteq\prod
F(x^n)$ lifts to an element lying in $\lim\CT(x^n,t)$ along the
natural map $\prod\CT(x_n,t)\to\prod F(x^n)$ which corresponds to
$b$ via the Yoneda isomorphism, and the proof of our claim is
complete.
\end{proof}

Recall that we say that $\CT$ satisfies the Brown representability
theorem if every cohomological functor $F:\CT\to\Ab$ which sends
coproducts into products is representable.

\begin{theorem}\label{brtf} Let $\CT$ be a triangulated category with
coproducts which is $\aleph_1$-perfectly generated by a projective
class $(\CP,\FF)$. Suppose also that every category $\CP^{*n}/F$
has a weak terminal object, for every $n\in\N$ and every
cohomological functor which sends coproducts into products
$F:\CT\to\Ab$. Then $\CT$ satisfies the Brown representability
theorem.
\end{theorem}

\begin{proof} It is shown in \cite[Theorem 1.3]{NR} that $\CT$ satisfies the
Brown representability theorem if and only if every cohomological
functor $F:\CT\to\Ab$ which sends coproducts into products has a
solution object, or equivalently, the category $\CT/F$ has a weak
terminal object. Thus the conclusion follows from this result
corroborated with Proposition \ref{wterm}.
\end{proof}

We will say that $\CT$ is $\aleph_1$-perfectly generated by a set
if it is $\aleph_1$-perfectly generated by a the projective class
induced by that set, in the sense above. Thus Theorem above
together with Lemma \ref{wtermpi} give:

\begin{corollary}\label{brt}
Let $\CT$ be a triangulated category with coproducts which is
$\aleph_1$-perfectly generated by a set. Then $\CT$ satisfies the
Brown representability theorem.
\end{corollary}

\begin{remark}\label{hk}
Our condition $\CT$ to be $\aleph_1$-perfectly generated by a set
is obviously equivalent to the hypothesis of \cite[Theorem A]{KB}.
Therefore Corollary \ref{brt} is the same as \cite[Theorem A]{KB},
but with a completely different proof. Note also that every
well--generated triangulated category in the sense of Neeman
\cite{N} is perfectly generated  by a set, in the sense above, as
it is shown in \cite{KN}.
\end{remark}

\end{document}